\documentclass[a4paper,12pt,reqno]{amsart}

\usepackage{amsmath,amssymb,pb-diagram,enumerate,euscript,endnotes,verbatim,graphicx,amscd,latexsym,amsfonts,mathrsfs}

\theoremstyle{plain}
\newtheorem{theorem}{Theorem} [section]    

\newtheorem{corollary}[theorem]{Corollary}

\theoremstyle{definition}
\newtheorem{definition}[theorem]{Definition}
\newtheorem{examples}[theorem]{Examples}

\theoremstyle{definition}
\newtheorem{remark}[theorem]{Remark}
\newtheorem{remarks}[theorem]{Remarks}

\newcommand{\R}{\mathbf{R}}
\newcommand{\C}{\mathbf{C}}

\def\lin{{\rm lin\,}}
\def\pr{{\rm pr}}
\def\ran{{\rm ran\,}}

\newcommand{\la}{\langle}
\newcommand{\ra}{\rangle}
\newcommand{\lra}{\longrightarrow}
\newcommand{\restrict}[2]{{#1}\!\left|_{#2}\right.}
\DeclareMathOperator{\esssup}{ess \, sup}
\DeclareMathOperator{\spec}{Sp}

\begin{document}

\newenvironment{enumeratei}{\begin{enumerate}[\upshape (i)]}{\end{enumerate}}
\newenvironment{enumeratea}{\begin{enumerate}[\upshape (a)]}{\end{enumerate}}

\title[]{Involutions and Trivolutions in Algebras Related to Second Duals of Group Algebras}
\author{M. Filali}
\address{Department of Mathematical Sciences, University of Oulu, Oulu 90014, Finland}

\email{mahmoud.filali@oulu.fi}

\author{M. Sangani Monfared}
\address{Department of Mathematics and  Statistics, University of Windsor, Windsor, ON, N9B~3P4, Canada.}

\email{monfared@uwindsor.ca}

\author{Ajit Iqbal Singh}
\address{Theoretical Statistics and Mathematics Unit, Indian Statistical Institute, 7, S.J.S. Sansanwal Marg, New Delhi-110\  016, India}

\email{aisingh@isid.ac.in, aisingh@sify.com}




\dedicatory{}

\maketitle
\begin{abstract}
We define a trivolution on a complex algebra $A$ as a non-zero conjugate-linear, anti-homomorphism $\tau$ on $A$, which is
a  generalized inverse of itself, that is,  $\tau^3=\tau$.
We give several characterizations of trivolutions and show with examples that they appear naturally on many Banach algebras,
particularly those arising from group algebras.
We give several results on the existence or non-existence of involutions on the dual of a topologically introverted space.
We  investigate conditions under which the dual of a topologically introverted space admits trivolutions.
\end{abstract}
\section{Introduction and preliminaries}

By a well-known result of Civin and Yood \cite[Theorem 6.2]{CiYo61},  if $A$ is a Banach algebra with an involution
$\rho \colon A \lra A$,  then the second (conjugate-linear) adjoint $\rho^{**}\colon A^{**}\lra A^{**}$  is an involution on $A^{**}$
(with respect to either of the Arens products)  if and only if $A$ is Arens regular; when this is the case, $\rho^{**}$ is called the canonical extension of $\rho$.   Grosser \cite[Theorem 1]{Grosser84}
has shown  that if $A$ has a bounded right [left] approximate identity,
then a necessary condition for the existence of an involution on $A^{**}$ with respect to the first [second] Arens product
is that $A^*\cdot A=A^*$ [$A\cdot A^*=A^*$].  The above results applied to the group algebra $L^1(G)$ imply that
 a necessary condition for $L^1(G)^{**}$ to have an involution (with respect to either of the Arens products) 
is that  $G$ is discrete (Grosser \cite[Theorem 2]{Grosser84}); moreover, 
the natural involution of $L^1(G)$ has the canonical extension  to $L^1(G)^{**}$ if and only if $L^1(G)$ is Arens regular,
and, by a well-known result of Young \cite{Young73}, $L^1(G)$ is Arens regular if and only if  $G$ is finite.

In general,  since  $A$ is not  norm dense in $A^{**}$, an involution on $A$  may have  extensions
to $A^{**}$ which are different from the canonical extension. A necessary and sufficient condition for the existence of such
extensions does not seem to be known. However, for the special case of the group algebra $L^1(G)$,
Farhadi and Ghahramani \cite[Theorem 3.2(a)]{FaGh07} have shown that
if a locally compact group $G$ has an infinite amenable subgroup, then $L^1(G)^{**}$ does not
have any involution extending the natural involution of $L^1(G)$. (See also the related paper by Neufang \cite{Neufang10},
answering a question raised in  \cite{FaGh07}.)
 It is unknown whether for a discrete
 group which does not contain an infinite
 amenable group, there is an involution on $L^1(G)^{**}$ extending Êthe natural involution of $L^1(G)$.Ê

In  Singh \cite{Singh11}, the third author introduced the concept of $\alpha$-amenability for a locally compact group $G$.
Given a cardinal $\alpha$,  a group $G$ is called $\alpha$-amenable if there exists a subset $\mathcal F \subset L^1(G)^{**}$
containing a mean $M$ (not necessarily left invariant)
such that $|\mathcal F|\le \alpha$ and the linear span of $\mathcal F$ is a left ideal of $L^1(G)^{**}$.
The group $G$ is called subamenable if $G$ is $\alpha$-amenable for some cardinal  $1\le \alpha < 2^{2^{\kappa (G)}}$,
where $\kappa (G)$ denotes the compact covering number of $G$, that is, the least cardinality of a compact covering of $G$.
(We remark that at the time of the proof-reading 
of the paper Singh \cite{Singh11}, the author did not know that the term subamenable had already been used in different contexts,
viz., subamenable semigroups by Lau and Takahashi in \cite{LaTa96,LaTa03}, and initially subamenable groups
by Gromov in \cite{Gromov99}. She thanks A.\  T.-M. Lau for interesting exchange of mathematical points in this regard.)
It follows that $1$-amenability of $G$ is equivalent to the amenability of $G$, and $\alpha$-amenability implies $\beta$-amenability for every $\beta \ge \alpha$. If $G$ is a non-compact locally compact group, then every non-trivial right ideal in $L^1(G)^{**}$ or  in $LUC(G)^*$ has  (vector space) dimension  at least $2^{2^{\kappa (G)}}$ (Filali--Pym \cite[Theorem 5]{FiPy03}, and, Filali--Salmi \cite[Theorem 6]{FiSa07}).
It follows from  this lower bound   that if $G$ is a subamenable, non-compact, locally compact group,
then $L^1(G)^{**}$ has no involution (Singh \cite[Theorem 2.2]{Singh11}).
Singh \cite[Theorem 2.9(i)]{Singh11} also showed that every discrete group $G$ is a subgroup of a subamenable discrete group
$G_\sigma$ with $|G_\sigma | \le 2^{|G|}$.  It is not known whether there exists any non-subamenable group, and in particular,
it remains an open question whether the free group on $2$ generators is  subamenable.

All the above results show that the existence of involutions on  second dual Banach algebras impose strong conditions
on $A$. So  it seems natural to consider involution-like operators on Banach algebras and their second duals.
In this paper, we relax the condition of bijectivity on an involution $\rho$ and of $\rho$ being its own inverse
to that of $\rho$ being a  generalized inverse of itself, namely, $\rho^3 = \rho,$ and call them trivolutions (Definition \ref{090112A})
following  the terminology of J. W. Degen \cite{Degen94}. It follows from the definition that every involution is a trivolution,
but as we shall show later there are many naturally arising trivolutions which are not involutions.

In section \ref{se:140312C}  we start  with a general study of trivolutions on  algebras and we
give several characterizations of trivolutions in Theorem \ref{th:070911B}.  We  use the decompositions 
given in Theorem \ref{th:070911B}(iii) to characterize trivolutive homomorphisms (Theorem \ref{th:120712B}).
In Theorem \ref{th:041111A} we show
that, unlike involutions, a trivolution can have various extensions to the unitized algebra $A^{\sharp}$, and we give a
complete characterization of all such extensions.  
We show that concepts such as hermitian, normal, and unitary elements, usually associated with  involutions,
can be naturally defined in the context of trivolutive algebras.  In section \ref{se:140312A}, we study involutions on the dual
of a topologically introverted space. In Theorem \ref{le:150212A} we extend the result of Civin and Yood (discussed above)
to the dual of a topologically introverted space, and  obtain the result of Farhadi and Ghahramani \cite[Theorem 3.2(a)]{FaGh07}
as a corollary.  In Theorem \ref{le:231111A}  we investigate the relationship between the existence of topologically
invariant elements and the existence of involutions on the dual of a topologically introverted space.
As a corollary we show  that under fairly general conditions, neither of the Banach algebras  
$PM_p(G)^*$ and $UC_p(G)^*$, $1<p<\infty$,  have involutions
(for the definition of these spaces see below as well as the discussion prior to Corollary \ref{co:241111A}).
In  section \ref{se:140312B} we give some sufficient conditions under which
a second dual Banach algebra $A^{**}$ admits  trivolutions (Theorem \ref{th:300112B}).  In Theorem \ref{th:150911A}
we show that for $G$ non-discrete, $L^1(G)^{**}$ does not admit
any trivolutions with range  $L_0^\infty (G)^*$.  However, the space $L_0^\infty (G)^*$ itself always admits
 trivolutions (Theorem \ref{th:150911D}).

We close this section with a few preliminary definitions and notation.  Given a Banach algebra $A$, the dual space $A^*$ can be viewed as a Banach $A$-bimodule with the canonical operations:
\[
\la \lambda\cdot a, b\ra =\la \lambda ,ab\ra ,\quad \la a\cdot \lambda , b\ra =\la \lambda ,ba\ra ,
\]
where $\lambda \in A^*$ and $a,b\in A$. Let $X$ be a norm closed $A$-submodule of $A^*$. Then given $\Psi \in X^*$, $\lambda \in X$, we may define
$\Psi \cdot \lambda \in A^*$ by $\la \Psi\cdot \lambda , a\ra =\la \Psi , \lambda \cdot a \ra $. If $\Psi \cdot \lambda \in X$
for all choices of $\Psi \in X^*$ and $\lambda \in X$, then $X$ is called a left topologically  introverted subspace of $A^*$.
The dual of a left topologically  introverted subspace $X$ can be turned
into a Banach algebra if, for all $\Phi, \Psi \in X^*$, we define $\Phi \Box \Psi \in X^*$ by
$\la \Phi \Box \Psi , \lambda \ra =\la \Phi , \Psi \cdot \lambda \ra$.  In particular, by taking $X=A^*$, we obtain
the first (or the left) Arens product on $A^{**}$, defined by Arens \cite{Arens51,Arens51b}.
The space $X^*$ can be identified with the quotient algebra  $A^{**}/X^\circ$,
where $X^\circ =\{\Phi \in A^{**}\colon \restrict{\Phi}{X} =0\}$.
If $X$ is faithful (that is, $a=0$ whenever $\lambda (a)=0$ for all $\lambda \in X$),
then the natural map of $A$ into $X^*$ is an embedding, and we will regard $A$ as a subalgebra of $(X^*,\Box)$.
The space $X^*$ has a canonical $A$-bimodule structure defined by $\la a\cdot \Phi ,\lambda \ra =\la \Phi , \lambda \cdot a\ra$,
$\la \Phi \cdot a ,\lambda \ra =\la \Phi , a\cdot \lambda \ra$ $(\Phi \in X^*, \lambda \in X , a\in A)$.
One can then verify that $a\cdot \Phi =a\Box \Phi, \ \Phi \cdot a=\Phi \Box a$ for each $a\in A$ and $\Phi \in X^*$.
For each $\Phi\in X^*$, the map  $\Psi\mapsto \Psi \Box \Phi$, $ X^*\lra X^*$, is  continuous in the
$w^*$-topology of $X^*$.

Right topologically introverted subspaces
of $A^*$ are defined similarly; for these spaces the second (or the right) Arens product on $X^*$ will be denoted by $\Phi \Diamond \Psi$.
A space which is both left and right topologically introverted is called topologically introverted. When $A$ is commutative there will be
no distinction between left and right topologically introverted spaces.

Let $A$ be a Banach algebra. The space of [weakly] almost periodic
functionals on $A$,  denoted by $[WAP(A)]\ AP(A) $, is defined as the set of all $\lambda \in A^*$ such that the linear map
$A\lra A^*$, $a\mapsto a\cdot \lambda$, is [weakly] compact. The spaces $A^*$, $WAP(A)$, and $AP(A)$ are examples of
topologically introverted spaces. The space of left uniformly continuous functionals on $A$ defined by
$LUC(A)=\overline{\lin}(A^*\cdot A)$ (the closure is in norm topology), is an example of a left topologically introverted space.
If $G$ is a locally compact group, then $LUC(L^1(G))$ coincides with $LUC(G)$, the space of left uniformly continuous functions
on $G$ (cf.\ Lau \cite{Lau87}).  For more information and additional examples
one may consult  \cite{Dales00,DaLa05,DFS04,Granirer87,Lau79,SpSt}.

If $f\colon X \lra X$ is a map on a set $X$, for simplicity and when no confusion arises, we write $f^n$ to denote the $n$-times composition
of $f$ with itself, that is, $f^n :=f \circ \cdots \circ f$ $(n$-times$)$. Also when no confusion arises,
we write $\Phi \Psi$ to denote the first Arens product  $\Phi \Box \Psi$.
\section{Trivolutions}\label{se:140312C}

\begin{definition}\label{090112A}
A \emph{trivolution} on a complex algebra $A$ is a non-zero, conjugate linear, anti-homomorphism $\tau \colon A\lra A$, such that $\tau^3  =\tau$.
When $A$ is a normed algebra,  we  shall assume that $\|\tau\|= 1 $.  The pair $(A, \tau )$  will be called a \emph{trivolutive algebra}.
\end{definition}

\begin{remarks}\label{re:250312D}
(i) It follows from the definition that every involution on a non-zero complex algebra is a trivolution. Conversely, a trivolution which is either
injective or surjective, is an involution.
If $(A,\tau )$ is a trivolutive normed algebra, then $\|\tau (x)\| \le \|x\|$ for
every $x\in A$; in particular, when $\tau$ is an involution, we have $\|\tau (x)\|=\|x\|$ for every $x$.

(ii) If $(A,\tau )$ is a trivolutive normed algebra, then $\tau (A)$ is closed subalgebra of $A$:  if $(x_n)$ is a sequence
such that $\tau (x_n)\to x$ for some $x\in A$, then $\tau(x_n)=\tau^3(x_n) \to \tau^2(x)$, and hence $x=\tau ^2(x)\in \tau (A)$.
\end{remarks}

The next result gives  several characterizations of trivolutions. We remark that by the equivalence of (i) and (iii)
in the following theorem, it follows that every trivolutive [Banach] algebra is a [strongly] splitting extension of
an involutive [Banach] algebra. 
\begin{theorem}\label{th:070911B}
Let $A$ be a complex algebra, $\tau \colon A \lra A$  a non-zero map, and  $B:=\tau (A)$.  Let $\mathscr I_B$ denote the set of all involutions on $B$.
Then the following statements are equivalent.
\begin{enumeratei}
\item $\tau$ is a trivolution.
\item $\tau$ is a conjugate-linear, anti-homomorphism, and $\restrict{\tau}{B}\in \mathscr I_B$.
\item $B$ is a subalgebra of $A$ and there exist a projection  $p \colon A \lra B$ with $p$ an algebra homomorphism,
a two sided ideal $I$ of $A$, and an involution $\rho$ on $B$   such that 
\begin{equation}\label{eq:160712A}
A=I\oplus B, \quad \tau =\rho \circ p.  
\end{equation}
\item $B$ is a subalgebra of $A$,   $\mathscr I_B$ is non-empty, and for each  $\rho_2\in \mathscr I_B$,
there is a surjective homomorphism $\rho_1\colon A \lra B$  that satisfies $\tau =\rho_2\circ \rho_1$ and $\rho_1\circ \rho_2\in \mathscr I_B$.
\end{enumeratei}
The statements {\rm (i)--(iv)} remain equivalent if $A$ is a normed algebra  and  the maps
in {\rm (i)--(iv)} are assumed to be contractive. 
\end{theorem}

\begin{proof}
(i)$\,\Longrightarrow\,$(ii): Since $\tau$ is a conjugate-linear, anti-homomorphism,  it follows that $B$ is a subalgebra of $A$; moreover,
the identity  $\tau^3=\tau$ implies that $(\restrict{\tau}{B})^2$ is the identity map on $B$, and therefore $\restrict{\tau}{B}\in \mathscr I_B$.

(ii)$\,\Longrightarrow\,$(iii):  $B$ is a subalgebra of $A$ since $\tau$ is a conjugate-linear anti-homomorphism.
Let $\rho :=\restrict{\tau}{B}$ and $p :=\tau \circ \tau \colon A \lra B$.
We leave it for the reader to verify the easy facts that $\tau =\rho \circ p$,
and  $p$ and $\rho$ satisfy the requirements in (iii). If we let $I=\ker \tau =\ker p$, then $I$ is a
two-sided ideal of $A$, and since $p$ is a projection we have $A=I\oplus B$. Moreover, since $I=\ker \tau$ and
$\rho =\restrict{\tau}{B}$, it follows that for each $x=y+z\in I\oplus B$, we have $\tau (x)=\tau (y)+\tau (z)=\rho (z)$.

(iii)$\,\Longrightarrow\,$(iv):  The first two statements of (iv) are immediate consequences of (iii).
Let $\rho_2\in \mathscr I_B$ be given and define $\rho_1: =\rho_2\circ \tau .$ Then $\rho_1$ is a
surjective homomorphism from $A$ to $B$ and furthermore, $\rho_2\circ \rho_1=\rho_2\circ \rho_2\circ \tau =I_B\circ \tau=\tau $
(where $I_B$ denotes the identity map on $B$).
Since  $\rho_1\circ \rho_2$  is a conjugate-linear, anti-homomorphism on $B$, it remains to show that  $(\rho_1\circ \rho_2)^2=I_B$.
This follows from the fact that $\tau^2=p$, since then  
\[
(\rho_1\circ \rho_2)^2=\rho_2\circ \tau \circ \rho_2\circ \rho_2\circ \tau\circ \rho_2
=\rho_2\circ \tau^2 \circ \rho_2
=\rho_2\circ p\circ \rho_2=I_B.
\]

(iv)$\,\Longrightarrow\,$(i): Since $\tau =\rho_2\circ \rho_1$, with $\rho_1,\rho_2$ as in (iv), it follows that $\tau$ is a surjective,
conjugate-linear anti-homomorphism. Now using the assumption that $\rho_1\circ\rho_2$ is an involution on $B$, we have
\[
\tau^3=\rho_2\circ (\rho_1\circ\rho_2)\circ (\rho_1\circ\rho_2)\circ \rho_1=\rho_2\circ \rho_1=\tau ;
\]
hence $\tau$ is a trivolution.
\end{proof}

\begin{corollary}\label{co:221012A}
 Suppose that $I=\ker \tau \ne \{0\}$.  Let $B=\tau (A)$ and $A=I\oplus B$ as in \eqref{eq:160712A}.  Then the following hold.
\begin{enumeratei}
\item $\tau =\begin{pmatrix}
0 & 0\\
0 & \rho
\end{pmatrix},$ where $\rho =\restrict{\tau}{B}.$
\item If $I$ has an involution $J$, then there are an algebra $C$, a homomorphism $\lambda \colon A \lra C$,
an involution $\sigma$ on $C$, and a homomorphism $\mu \colon C \lra A$ such that the  diagram 
\[
\begin{CD}
A @> \tau >>A\\
@V{\lambda} VV   @AA{\mu}A\\
C @>>\sigma > C  
\end{CD} 
\]
commutes, namely, $\tau =\mu \circ \sigma \circ \lambda.$
\end{enumeratei}
\end{corollary}

\begin{proof}
Statement (i) follows immediately from the fact that for each $x=y+z\in I\oplus B$, we have $\tau (x)=\tau (y)+\tau (z)=\rho (z)$.

To prove (ii), let $\pr_1 =I_A-\tau^2$ be the projection of $A$ onto $I$. We define the algebra  $\mathcal A$ as
the vector space $I\oplus B$ equipped  with coordinatewise multiplication, therefore, if $x=y+z\in \mathcal A$ and  $x'=y'+z' \in \mathcal A$ (with
$y,y'\in I$ and  $z,z'\in B$), then
 $xx'=yy'+zz'\in \mathcal A$. Now we let $C=\mathcal A \times \mathcal A$, equipped with
coordinatewise operations, and  we define
\[
\lambda (x)=(\tau^2(x),0), \qquad \mu (x,x')=\tau^2(x),
\]
and 
\begin{align*}
\sigma (x,x') &=
\begin{pmatrix}
\tau & J\circ \pr_1\\
J\circ \pr_1 & \tau
\end{pmatrix}
\begin{pmatrix}
x\\
x'
\end{pmatrix}\\
&=(J\circ \pr_1(x')+\tau (x), J\circ \pr_1(x)+\tau (x')).
\end{align*}
It is easy to check that both $\lambda$ and $\mu$ are homomorphisms, $\sigma$ is
a conjugate-linear anti-homomorphism, and   $\tau =\mu \circ \sigma \circ \lambda$.   Moreover, since
\begin{align*}
\sigma^2&=
\begin{pmatrix}
\tau^2+J\circ\pr_1\circ J\circ \pr_1 & \tau\circ J\circ \pr_1 +J\circ \pr_1\circ \tau\\
\tau\circ J\circ \pr_1 +J\circ \pr_1\circ \tau  & \tau^2+J\circ\pr_1\circ J\circ \pr_1
\end{pmatrix}\\
&=
\begin{pmatrix}
\tau^2+\pr_1 & 0\\
0  & \tau^2+\pr_1
\end{pmatrix}
=
\begin{pmatrix}
I_A & 0\\
0  & I_A
\end{pmatrix},
\end{align*}
it follows that $\sigma$ is an involution on $C$, completing the proof of (b).
\end{proof}

If $(A,\tau )$ is a trivolutive algebra, then we shall call the identities  in \eqref{eq:160712A} of Theorem \ref{th:070911B}, namely, 
$A=I\oplus B$, $\tau =\rho\circ p$, 
in which $p=\tau^2$, $B=p(A)=\tau (A)$, 
$I=\ker p=\ker \tau$, and $\rho=\restrict{\tau}{B}$, the \emph{canonical decompositions} of $(A, \tau )$.
As an interesting application of these ideas, we shall give a characterization of trivolutive
homomorphisms $\pi \colon (A_1, \tau_1)\lra (A_2, \tau_2)$. These are  algebra homomorphisms 
such that $\pi (\tau_1(x))=\tau_2(\pi (x))$, for all $x\in A_1$.

\begin{theorem}\label{th:120712B}
Let $(A_i,\tau_i)$, $i=1,2$, be two trivolutive [Banach] algebras and $\pi \colon A_1\lra A_2$ be a [continuous] trivolutive homomorphism.
Let $A_i=I_i\oplus B_i$ and $\tau_i=\rho_i\circ p_i$, be the canonical decompositions of $(A_i,\tau_i)$.
Then 
\[
\pi =
\begin{pmatrix}
\pi_{11} & 0\\
0 & \pi_{22}
\end{pmatrix}, 
\]
where $\pi_{11}\colon I_1\lra I_2$ and $\pi_{22}\colon B_1\lra B_2$ are both  [continuous] homomorphisms, and $\pi_{22}$
is involutive.  
\end{theorem}

\begin{proof}
We confine our attention to the case when $I_1$ and $I_2$ are both
non-zero.
Following \eqref{eq:160712A} of Theorem \ref{th:070911B}, we can write 
$\tau_1 =
\begin{pmatrix}
0 & 0\\
0 & \rho_1
\end{pmatrix}$, and 
$\tau_2 =
\begin{pmatrix}
0 & 0\\
0 & \rho_2
\end{pmatrix}$.
Let
$
\pi =
\begin{pmatrix}
\pi_{11} & \pi_{12}\\
\pi_{21} & \pi_{22}
\end{pmatrix}
$
be the block matrix representation of $\pi$. The identity $\pi\circ \tau_1=\tau_2\circ \pi$ implies that
\[
\begin{pmatrix}
0 & \pi_{12}\circ \rho_1\\
0 & \pi_{22}\circ \rho_1
\end{pmatrix}= \begin{pmatrix}
0 & 0\\
\rho_2\circ\pi_{21} & \rho_2\circ\pi_{22}
\end{pmatrix}.
\]
Since both $\rho_1$ and $\rho_2$ are bijective, the identities $\pi_{12}\circ\rho_1=0$ and $\rho_2\circ \pi_{21}=0$,
imply that $\pi_{12}=0$ and $\pi_{21}=0$.  Now it is easy to check  that both $\pi_{11}$ and $\pi_{22}$ must be algebra homomorphisms. 
Moreover,  the identity $\pi_{22}\circ\rho_1=\rho_2\circ \pi_{22}$,
implies that $\pi_{22}\colon (B_1, \rho_1) \lra (B_2,\rho_2)$ is  involutive, as we wanted to show. 
\end{proof}

There are various automatic continuity results for homomorphisms and for homomorphisms intertwining with involutions in the
literature (cf.\ Dales \cite{Dales00} and Palmer \cite{Palmer94,Palmer01}). We can use these results to prove various automatic continuity results involving trivolutions. We illustrate just one such result.
By a famous result of B.\ E.\ Johnson if $A$ and $B$ are Banach algebras with $B$ semisimple, then any
surjective homomorphism $\varphi \colon A \lra B$ is continuous (Palmer \cite[Theorem 6.1.3]{Palmer94}). 
Moreover, it is well known that if $A$ is an 
involutive Banach algebra and $B$ is a C$^*$-algebra, then every involutive algebra homomorphism 
$\pi \colon A \lra B$ is continuous (Dales \cite[Corollary 3.2.4]{Dales00}). 
As an immediate consequence of these results and Theorem \ref{th:120712B}, we can state the following automatic continuity  result
for trivolutive homomorphisms. 
\begin{corollary}
Let $(A_i, \tau_i)$, $i=1,2$, be two trivolutive Banach algebras and $\pi \colon A_1\lra A_2$ be a trivolutive  homomorphism. 
If $\ker \tau_2$ is semisimple and is contained in  $ \pi (A_1)$, then  $\pi$ is continuous provided that for every $b\in \tau_2(A_2)$
we have $\|\tau_2(b)b\|=\|b\|^2$.
\end{corollary}

Next, we give a few  examples of trivolutions.
\begin{examples}\label{ex:120911A}
(a)  Let $(X, \mu )$ be a measure space and $K\subset X$ be a measurable subset of $X$ with $\mu (K)>0$. 
Let $L_K^\infty (X, \mu)$ be the subalgebra
of $L^\infty (X, \mu )$ consisting of all those functions which vanish locally almost everywhere on $X \backslash K$.  Let $\chi_K$ be the
characteristic function of $K$ and  $p \colon L^\infty (X, \mu ) \lra L^\infty_K (X, \mu )$ be the homomorphism $f\mapsto \chi_K f$.
If  $\rho$ is the usual complex conjugation on $L^\infty_K(X, \mu )$,
then  the map $\tau$ given by
  $\tau (f):=\rho  \circ p (f)=\chi_K\overline{f}$ is  a trivolution on $L^\infty (X, \mu )$.
  We next give an abstract alternative.

(b) Let $H$ be a Hilbert space, $X$ a non-zero closed subspace of $H$, and  $P$ be the orthogonal projection on $X$.
Let $M$ be a von Neumann algebra on $H$ such that $P\in M'$, where $M'$ is the commutant of $M$. Let $N$ be
the von Neumann subalgebra of $M$ defined by $N=\{PT\colon T\in M\}$. Let $p: M\lra N$, be defined by
$p (T)=PT$, and let $\rho $ be the natural adjoint map on $N$. In that case, $p\circ \rho =\rho $, and
by Theorem \ref{th:070911B}, $\tau (T):=\rho \circ p (T)=PT^*$ defines a trivolution on $M$.

(c) The quotient of a trivolutive [normed] algebra $(A,\tau )$ by a two-sided [closed] ideal $I$ such that $\tau (I)\subset I$
and $\tau (A) \not\subset I$,  is  a trivolutive [normed] algebra.
Finite products of trivolutive [normed] algebras,
the completion of a trivolutive normed algebra, and the opposite of a trivolutive [normed] algebra, are all trivolutive [normed]
algebras in  canonical ways.
\end{examples}

\begin{theorem}\label{ex:120911B}
Let $\tau$ be an anti-homomorphism  on an algebra $A$ and let $B=\tau (A)$.
\begin{enumeratei}
\item  If $e\in B$ is a right  identity of $A$, then $\tau (e)=e$ and $e$ is the identity of $B$.
\item The set $B$ can contain at most one right  identity of $A$.
\item Let $A$ be a subalgebra of an algebra $C$ of the form $eC$ with $e$ being a right identity of $C$. If $\tau$ is
a trivolution on $A$, and if $\ell_e$
denotes the left  multiplication map by $e$ on $C$, then $\tau_1 :=\tau \circ \ell_e$ is a trivolution on $C$
and $\tau_1(C)=\tau (A)=B$.
\end{enumeratei}
\end{theorem}

\begin{proof}
(i) If $a\in A$, then $a=ae$ and hence $\tau (a)=\tau (e)\tau (a)$, which shows that $\tau (e)$ is a left identity
for $B$.  Since $e\in B$, $e=\tau (e)e=\tau (e)$, proving that $e=\tau (e)$ and $e$ is the identity for $B$.

(ii) This is an immediate consequence of (i).

(iii) Since $e$ is a right identity of $C$ and $B\subset A=eC$, it follows that $\restrict{\ell_e}{B}$ is the  identity
map on $B$: in fact, given $b\in B$,  we can write $b=ec$ for some $c\in C$, and hence $ \ell_e(b)=e(ec)=ec=b.$
It follows that $\ell_e\circ \tau =\tau$,  and hence
\[
\tau_1^3=(\tau\circ \ell_e)^3=\tau \circ (\ell_e \circ \tau )\circ (\ell_e\circ \tau )\circ \ell_e =\tau^3\circ \ell_e =\tau \circ \ell_e =\tau_1.
\]
In addition, since $e$ is a right identity of $C$, we have $e(c_1c_2)=(ec_1)(ec_2)$, for all $c_1,c_2\in C$. Hence $\ell_e$ is a
homomorphism on $C$ which implies that $\tau_1=\tau\circ \ell_e$ is a conjugate-linear, anti-homomorphism, completing the proof that
$\tau_1$ is a trivolution on $C$. The fact that $\tau_1 (C)=B$ is now immediate.
\end{proof}

\begin{remarks}
(i)  Similar results hold if a right identity is replaced by a left identity in Theorem \ref {ex:120911B};
we leave the formulation of the results and their  proofs for the readers.

(ii) Let $\tau$ be a trivolution on $A$ and let $B= \tau (A)$. If $A$ has the identity $e$, then $\tau (e)$
is the identity of $B$, which we may denote  by $e_B$. Clearly $e=e_B$ if and only if $e\in B$. This
however may not always be the case: let $A=\C^2$, $B=\C \times \{0\}$,
and $\tau (z_1,z_2)=(\overline z_1,0)$; then $e=(1,1)$ but $e_B=(1,0)$. 
\end{remarks}

Next we consider the problem of extending a trivolution to the unitized algebra $A^\sharp =\C \times A$.
Let $(A,\tau )$ be a trivolutive algebra and $\tau^\sharp \colon A^\sharp \lra A^\sharp$ be a trivolution extending $\tau$,
namely,  $\tau^\sharp (0,x)=(0,\tau (x)),\ \text{for all }x\in A$. If $\tau^\sharp (1,0)=(\lambda_0, x_0)$,
then, from conjugate linearity of $\tau^\sharp$ we obtain:
\begin{equation}\label{eq:250312E}
\tau^\sharp (\lambda , x)=\tau^\sharp (\lambda (1,0)+(0,x))=(\overline \lambda \lambda_0, \overline \lambda x_0+\tau (x)).
\end{equation}
 If $(\lambda_0, x_0)=(1,0)$, then $\tau^\sharp (\lambda , x)=(\overline \lambda , \tau (x))$.
We call this map the canonical extension of $\tau$ to $A^\sharp$. We note that by Theorem \ref{ex:120911B}(i),
the condition $\tau^\sharp (1,0)=(1,0)$ is equivalent to $(1,0)$ being in the range of $\tau^\sharp$, and
the latter condition can be shown to be equivalent to $\lambda_0\ne 0$ and $x_0\in \tau (A)$.

While every involution has only the canonical  extension to an involution on the unitized algebra,
the situation is different for trivolutions, as the following theorem shows.

\begin{theorem}\label{th:041111A}
Let $\tau$ be a trivolution on a complex algebra $A$.
The map $\tau^\sharp$ in \eqref{eq:250312E} is a trivolution extending $\tau$  if and only if either of the following
conditions hold:
\begin{enumeratei}
\item $\tau^\sharp (\lambda , x)=(\overline \lambda, \overline \lambda x_0+\tau (x))$, where $x_0\in A$ is such that
\begin{equation}\label{eq:041111B}
x_0^2=-x_0, \quad x_0\tau (A)=\tau (A)x_0=\{0\}, \quad \tau (x_0)=0.
\end{equation}
\item $\tau^\sharp (\lambda , x)=(0, \overline \lambda x_0+\tau (x))$, where $x_0\in \tau (A)$ is the identity of $\tau (A)$.
\end{enumeratei}
\end{theorem}

\begin{proof}
The proof that both  (i) and (ii) define trivolutions on $A^\sharp$ extending $\tau$, is routine and is left for the reader.
We prove the necessity part of the theorem.  Using the idempotence of $(1,0)$ we get
\[
(\lambda_0, x_0)=\tau^\sharp (1,0)=\tau^\sharp (1,0)^2=(\lambda_0^2, 2\lambda_0x_0+x_0^2),
\]
which implies that either $\lambda_0=1$ and $x_0^2=-x_0$; or $\lambda_0=0$ and $x_0^2=x_0$.
We consider these two cases.

Case I: $\lambda_0= 1$ and $x_0^2=-x_0$.
In this case $\tau^\sharp (\lambda , x)=(\overline \lambda, \overline \lambda x_0+\tau (x))$.
Applying $\tau^\sharp$ to the identity $(1,0)(0,x)=(0,x)$, we obtain
\[
(0,\tau (x)+\tau (x)x_0)=(0, \tau (x)),
\]
which implies that $\tau (x)x_0=0$ for all $x\in A$. Similarly, starting from the identity $(0,x)(1,0)=(0,x)$ we can show that
$x_0\tau (x)=0$ for all $x\in A$.  Moreover
it follows from $(\tau^\sharp) ^3 (1,0)=\tau^\sharp (1,0)$, that
\[
(1, x_0+\tau (x_0)+\tau^2 (x_0))=(1,x_0),
\]
which is equivalent to $\tau (x_0)+\tau^2(x_0)=0$. Therefore
\[
0=x_0\tau (x_0)=\tau^2(x_0)\tau (x_0)=-\tau (x_0)^2=-\tau (x_0^2)=\tau (x_0).
\]
Thus $x_0$ satisfies all the conditions in \eqref{eq:041111B}.

Case II: $\lambda_0=0$. In this case $\tau^\sharp (\lambda , x)=(0, \overline \lambda x_0+\tau (x))$.
Applying $\tau^\sharp$ to the identities $(1,0)(0,x)=(0,x)$ and $(0,x)(1,0)=(0,x)$, we  obtain respectively,
$\tau (x)x_0=\tau (x)$, $x_0\tau (x)=\tau (x)$,  for all $x\in A$.
Moreover, from $(\tau^\sharp) ^3 (1,0)=\tau^\sharp (1,0)$, it follows that $x_0=\tau^2 (x_0)\in \tau (A)$. Thus $x_0$
is the identity of $\tau (A)$.
\end{proof}

\begin{corollary}\label{co:171011A}
Let $(A, \tau )$ be a trivolutive normed algebra and  $A^\sharp$ be the unitized algebra with the norm $\|(\lambda , x)\|=|\lambda |+\|x\|$.
A map $\tau^\sharp\colon A^\sharp \lra A^\sharp$,
is a trivolution extending $\tau$  if and only if either of the following
conditions hold:
\begin{enumeratei}
\item $\tau^\sharp (\lambda , x)=(\overline \lambda, \tau (x))$;
\item $\tau^\sharp (\lambda , x)=(0, \overline \lambda x_0+\tau (x))$, where $x_0\in \tau (A)$
is the identity of $\tau (A)$ with  $\|x_0\|=1$.
\end{enumeratei}
\end{corollary}

\begin{proof}
If $\tau^\sharp$ is of the form given in Theorem \ref{th:041111A}(i), then $\tau^\sharp (\lambda , 0)=(\overline\lambda , \overline\lambda x_0)$
for all $\lambda \in \C$.  Since we must have $\|\tau^\sharp\|= 1$, we obtain  $|\lambda|+|\lambda |\|x_0\|\le |\lambda |$,
implying that $x_0=0$.

If however $\tau^\sharp$ is of the form given in Theorem \ref{th:041111A}(ii), then
$\tau^\sharp (1 , 0)=(0 ,  x_0)$. Hence the condition
$\|\tau^\sharp\|= 1$ implies that $\|x_0\|\le 1$. The inequality $\|x_0\|\ge 1$
is immediate since $x_0$ is a non-zero idempotent.
\end{proof}

\begin{remark}
Let $\tau$ be an involution on a complex algebra $A$. Then $\tau (A)=A$, and hence any extension of $\tau$ of the form given in Theorem \ref{th:041111A}(i),
is necessarily equal to $\tau^\sharp (\lambda , \tau (x))=(\overline\lambda , \tau (x))$, since $x_0\tau (A)=\{0\}$ implies
that $x_0^2=0$ and hence $x_0=0$ (as $x_0^2=-x_0$). It should be noted that if $A$ has no identity, then $\tau$ has no extension
of the form given in  Theorem \ref{th:041111A}(ii).
\end{remark}

We can define the concepts of normality, hermiticity, and positivity for elements of trivolutive algebras.

\begin{definition}
Let $(A, \tau )$ be a trivolutive algebra and let $x\in A$. Then $x$ is called
\begin{enumeratei}
\item  \emph{hermitian} if $\tau (x)=x$;
\item  \emph{normal} if $x\tau (x)=\tau (x)x$ and $x\tau^2 (x)=\tau ^2(x)x$;
\item  \emph{projection} if  $x$ is hermitian and  $x^2=x$;
\item \emph{unitary} if $A$ is unital with identity $e$ and $x\tau (x)=\tau (x)x=e$;
\item \emph{positive} if  $x$ is hermitian and $x=\tau (y)y$ for some $y\in A$.
\end{enumeratei}
\end{definition}

We denote the set of all  hermitian (respectively, unitary, positive) elements of $A$,  by $A_h$ (respectively, $A_u$,  $A^+$).
It follows that $A_h$ is a real vector subspace of $A$, and  $A_u$ is a group under multiplication (the unitary group of $A$).
It follows from the definition that if $x$ is hermitian, then  $x\in A^+$  if and only if $x=z\tau (z)$ for some $z\in A$.
It should be noted that  for trivolutive algebras in general, $A^+$ need not form  a positive cone.
The definition of normality is designed to have the $\tau$-invariant algebra generated by $x$ (and therefore, containing
both $\tau (x)$ and $\tau^2(x)$) commutative (since  $\tau (x)\tau^2 (x)=\tau (\tau (x)x)=\tau (x\tau (x))=\tau^2 (x)\tau (x)$).
If $x$ is unitary in $A$, then by letting $B=\tau (A)$ and $e_B=\tau (e)$, we see that $\tau^2(x)$ is the inverse of $\tau (x)$ in $B$,
and $x\in B$ implies that $e\in B$ (and  $e=e_B$). Thus, $e\in B$ if and only if $B$ contains at least one unitary element.

Let $(A,\tau )$ be a trivolutive algebra and  $\tau^*$ be the conjugate-linear  adjoint of $\tau$ 
defined by $\la \tau^* (f), a\ra =\overline{\la f, \tau (a)\ra}, \ (f\in A^*, a\in A)$.
If $f\colon A \lra \C$ is a linear functional on $A$, then $f^\tau :=\tau^* (f)$,  is also a linear
functional on $A$.  One can easily check that the map $f\lra f^\tau$, is  conjugate-linear and
in general  $f^{\tau\tau\tau}=f^\tau$.  If $(A, \tau )$ is normed, then $\|f^\tau \|\le \|f\|$.
We call $f$ hermitian if $f^\tau=f$. Clearly if $\chi$ is a (Gelfand) character on $A$, then $\chi^\tau$ is also a character on $A$.

We close this section by stating the following two results whose straightforward proofs are omitted
for brevity. 

\begin{theorem}\label{th:081011A}
Let $(A, \tau )$ be a unital trivolutive algebra and $B=\tau (A)$.    Let $x\in A$.
\begin{enumeratei}
\item  If  $x$ is invertible in $A$, then $\tau (x)$ is invertible in $B$ and $\tau (x)^{-1}=\tau (x^{-1})$.
\item If $\tau (x)$ is invertible in $A$, then $\tau (x)$ is invertible in $B$.
\item $\spec_B (\tau (x))\subset \overline{\spec_A (x)}$ (where the bar denotes the complex conjugate).
\end{enumeratei}
\end{theorem}

\begin{theorem}\label{th:081011A}
Let $(A,\tau )$ be a trivolutive algebra.
\begin{enumeratei}
\item  $x\in A$ can be written uniquely in the form $x=x_1+ix_2$, with $x_1,x_2$  hermitian, if and only if
$x\in \tau (A)$.
\item $f\in A^*$ can be written uniquely in the form $f=f_1+i f_2$, with $f_1,f_2$ hermitian, if and
only if $f\in \tau^*(A^*).$
\item A linear functional $f$ is hermitian if and only if $f$ is real valued on $A_h$ and it vanishes on $\ker \tau$.
\item The map $f \lra \restrict{f}{A_h}$ is an isomorphism
between the real vector space of all hermitian linear functionals and the dual vector space of the real space $A_h$.
\end{enumeratei}
\end{theorem}

\section{Involutions on the dual of a topologically introverted space}\label{se:140312A}
The following theorem  is an extension of a result of Civin and Yood \cite[Theorem 6.2]{CiYo61}
to the dual of a topologically introverted space.

\begin{theorem}\label{le:150212A}
Let $A$ be a Banach algebra and $X$, a faithful, topologically (left and right) introverted  subspace of $A^*$.
\begin{enumeratei}
\item If there is a $w^*$-continuous, injective, anti-homomorphism (with respect to either of the Arens products)
$\Theta \colon X^*\lra X^*$ such that $\Theta (A)\subset A$, then the two Arens
products coincide on $X^*$.
\item Let $\theta \colon A\lra  A$   be an involution on $A$ and let $\theta^*\colon A^*\lra A^*$ be its conjugate-linear adjoint.
If $\theta^* (X)\subset X$ and if the two Arens products coincide on $X^*$, then $\Theta=(\restrict{\theta^*}{X})^*\colon X^*\lra X^*$,
is an involution on $X^*$, extending $\theta$.
\end{enumeratei}
\end{theorem}

\begin{proof}
(i) Let $\mu , \nu \in X^*$, and let $(a_\alpha ), (b_\beta )$  be two nets in $A$ such that $a_\alpha \to \mu$, $b_\beta \to \nu$,
in the $w^*$-topology. Let us assume $\Theta$ is an anti-homomorphism with respect to the first Arens product.
Then $\Theta (\alpha_\alpha )\to \Theta (\mu )$, and $\Theta (b_\beta )\to \Theta (\nu )$.
Hence
\begin{align*}
\Theta (\mu \Box \nu )&=\Theta (\nu )\Box \Theta (\mu )\\
&=w^*\text{-}\lim_\beta \Theta (b_\beta )\Box \Theta (\mu )\\
&= w^*\text{-}\lim_\beta (w^*\text{-}\lim_\alpha \Theta (b_\beta )\Box \Theta (a_\alpha ))\\
&= w^*\text{-}\lim_\beta (w^*\text{-}\lim_\alpha \Theta (b_\beta ) \Theta (a_\alpha ))\\
&= w^*\text{-}\lim_\beta (w^*\text{-}\lim_\alpha \Theta (a_\alpha b_\beta ))\\
&= w^*\text{-}\lim_\beta  \Theta (\mu \Diamond b_\beta )\\
&=\Theta (\mu \Diamond \nu ),
\end{align*}
where $\Diamond$ denotes the second Arens product. Since $\Theta$ is injective, $\mu \Box \nu =\mu \Diamond \nu$, 
which is what we wanted to show.

The claim in (ii) follows by a similar argument as in (i).
\end{proof}

It is well known that the two Arens products coincide  on $WAP(A^*)^*$ and $AP(A^*)^*$ (Dales--Lau \cite[Proposition 3.11]{DaLa05}).
 It is also straight forward
to check that both of these spaces are  invariant under the conjugate-linear adjoint of any involution of $A$.
Therefore if either of these spaces is faithful (which is the case, for example, if the spectrum of $A$ separates the points of $A$;
see Dales--Lau \cite[p.\ 32]{DaLa05}), then  its dual has an involution extending that of $A$.
Hence as a corollary of the above theorem we obtain the following result due to
Farhadi and Ghaharamani (\cite[Theorem 3.5]{FaGh07}).
\begin{corollary}\label{co:150212D}
Suppose that $A$ is an involutive Banach algebra and $X$ is either of the topologically introverted spaces
$AP(A^*), WAP(A^*)$. If $X$ is faithful, then $X^*$ has an involution extending the involution of $A$.
\end{corollary}

Let $X\subset A^*$ be a faithful, topologically left introverted subspace of $A^*$.
Let $\sigma (A)$ denote the spectrum of $A$ and let $\varphi \in \sigma (A)\cap X$.  We call an element $m\in X^{*}$  a $\varphi$-topological invariant mean ($\varphi$-TIM)
if $\la m, \varphi  \ra =1$ and $a\cdot m =m\cdot a =\varphi (a)m$ for all $a\in A$.

The following theorem is an extension of a result of Farhadi and Ghahramani \cite[Theorem 3.2(a)]{FaGh07} (see the introduction)
to the dual of topologically left introverted  spaces.

\begin{theorem}\label{le:231111A}
Let $A$ be a Banach algebra and $X\subset A^*$ be a faithful, topologically left introverted subspace of $A^*$.
Let $\varphi \in \sigma (A)\cap X$. If $X^*$ contains at least  two $\varphi$-TIMs,
then $X^{*}$ cannot have an involution $*$ such that $\varphi (a^*)=\overline{\varphi (a)}$
for every $a\in A$ .
\end{theorem}

\begin{proof}
Let us suppose that $X^{*}$ has an involution as in the statement of the theorem. 
Let $m\in X^{*}$ be an arbitrary  $\varphi$-TIM. Then we have
\begin{equation}\label{eq:251111A}
a\cdot m^*=m^*\cdot a=\varphi (a)m^*, \qquad  (a\in A).
\end{equation}
To prove the above identities, we note that for every $a\in A$:
\[
a\cdot m^*=(m\cdot a^*)^*=(\varphi (a^*)m)^*=(\overline{\varphi (a)}m)^*=\varphi (a)m^*.
\]
The other identity in \eqref{eq:251111A} is proved similarly. Using the $w^*$-continuity of the  product $n\Box m$ on the variable $n$,
it follows from \eqref{eq:251111A}   that
\begin{equation}\label{eq:251111B}
n\Box m^*=\la n , \varphi \ra m^*, \qquad (n\in X^{*}).
\end{equation}
Thus using  \eqref{eq:251111B} and the fact that $\la m , \varphi \ra =1$, we get
\begin{equation}\label{eq:251111C}
m=(m^*)^*=(m\Box m^*)^*=m\Box m^*=\la m, \varphi \ra m^*=m^*.
\end{equation}
Now if $m_1,m_2$ are two distinct $\varphi$-TIMs, then using \eqref{eq:251111B} and \eqref{eq:251111C}, we have
\[
m_1=m_2\Box m_1=(m_2\Box m_1)^*=m_1^*\Box m_2^*=m_1\Box m_2=m_2;
\]
which is a contradiction.
\end{proof}

To state a corollary of the above theorem, we first recall a few definitions.
Let $G$ be a locally compact group, $1<p<\infty$,  and  $\mathscr L(L^p(G))$ be the space of continuous linear operators on $L^p(G)$.  Let $\lambda_p \colon M(G)\lra \mathscr L(L^p(G))$,  $\lambda _p(\mu )(g)=\mu *g$,
where $\mu * g(x)=\int_G g(y^{-1}x)\, d\mu (y)$, be the left regular representation of $M(G)$ on $L^p(G)$. The space
$PM_p(G)$ is the $w^*$-closure of $\lambda_p(M(G))$ in $\mathscr L(L^p(G))$. This space is the dual of the Herz--Fig\`a-Talamanca
algebra $A_p(G)$, consisting of all functions $u\in C_0(G)$, such that $u=\sum_{i=1}^\infty g_i*\check f_i$,
where $f_i\in L^p(G), g_i\in L^{q}(G)$,  $1/p+1/q=1$,
and $\sum_{i=1}^\infty \|f_i\|_p\|g_i\|_{q}<\infty$ (Herz \cite{Herz73}).  When $p=2$, $A_2(G)$ and $PM_2(G)$ coincide
respectively, with the Fourier algebra $A(G)$ and the group von Neumann algebra $VN(G)$ studied by Eymard in \cite{Eymard64}.
In the following, for simplicity of notation, we denote $UC(A_p(G))$
by $UC_p(G)$; when $p=2$, this space is also denoted by $UC(\widehat  G)$ in the literature.

\begin{corollary}\label{co:241111A}
 Let $1<p<\infty $, $G$ be  a non-discrete locally compact group, and  $X=PM_p(G)$ or $X=UC_p(G)$.
 Let $e\in G$ be the identity of $G$. Then $X^*$ does not have any involution $*$ such that $ u^*(e)=\overline{u(e)}$,
for every $u\in A_p(G)$.
\end{corollary}

\begin{proof}
 Let  $\varphi_e \in \sigma (A_p(G))\cap X$ be the evaluation
 functional at $e$. Let $TIM(X^*)$ denote the set of all $\varphi_e$-TIMs on $X^*$.
  Granirer \cite[Theorem, p.~3400]{Granirer96} has shown that if $G$ is non-discrete, then  $|TIM (X^*)| \ge 2^{\frak c}$,
  where $\frak c$ is the cardinality of real numbers. 
Therefore our result follows from Theorem \ref{le:231111A}.
\end{proof}

\begin{remark}
For $p=2$, the cardinality of $TIM (PM_2(G))$ was  determined for second countable groups  by Chou \cite{Chou82},
and in  full generality by Hu \cite{Hu95}.
\end{remark}
  
Let $G$ be a locally compact group and $LUC(G)$ the space of all left uniformly continuous functions on $G$.
It is known, and easy to verify, that the natural restriction map $\pi \colon L^\infty (G)^*\lra LUC(G)^*$
is a continuous algebra homomorphism (with respect to the first Arens product). Using this fact we can prove
the following analogue of Singh's result \cite[Theorem  2.2]{Singh11} for the non-existence of involutions
on $LUC(G)^*$.  Our result  extends Farhadi--Ghahramani \cite[Theorem 3.2(b)]{FaGh07}, from amenable  to subamenable groups.

\begin{theorem}\label{th:260412A}
If $G$ is a non-compact subamenable group, then $LUC(G)^*$ has no involutions.
\end{theorem}

\begin{proof}
Since $G$ is subamenable, there exists a subset $\mathcal F$ of $L^\infty (G)^*$ containing a mean $m$ such that
$|\mathcal F|< 2^{2^{\kappa (G)}}$, and the linear span of $\mathcal F$ is a left ideal $J$ of $L^\infty (G)^*$.
Since $\pi \colon L^\infty (G)^*\lra LUC(G)^*$ (defined above)  is  a continuous
 homomorphism,  it follows that $\pi (J)$ is a non-trivial left ideal in $LUC(G)^*$; $\pi (J)$ is non-trivial
since $\pi (m)\ne 0$. Since the dimension of $\pi (J)$ is less than or equal to the dimension of $J$,
it follows that $LUC(G)^*$ has a non-trivial left ideal of dimension less than $2^{2^{\kappa (G)}}$.
By Filali--Pym \cite[Theorem 5]{FiPy03}, if $G$  is a non-compact locally compact group, every non-trivial right 
ideal of $LUC(G)^*$ has  dimension  at least $2^{2^{\kappa (G)}}$. It follows that for non-compact subamenable
groups $G$, $LUC(G)^*$ cannot  have any involutions.
\end{proof}

\section{Trivolutions on the duals of  introverted spaces}\label{se:140312B}

In Corollary \ref{co:241111A} and  Theorem \ref{th:260412A}  we saw several examples of topologically
left introverted spaces $X$ for which there can be no involution on $X^{*}$. Our objective in  this section is to consider
some cases for which  $A^{**}$ or $X^*$  admits trivolutions.

\begin{theorem}\label{th:300112B}
Let $A$ be a non-zero Banach algebra with an involution $\theta$. Under each of the following conditions, $A^{**}$ 
admits a trivolution.
\begin{enumeratei}
\item There exists a non-zero topologically introverted, faithful subspace
$X\subset A^*$, such that the two Arens products coincide on $X^*$, $\theta^*(X)\subset X$,  and $A^{**}=X^\circ \oplus X^*$.
\item  $A$ is a dual Banach algebra.
\item $A$ has a bounded two-sided approximate identity and is a right ideal in $(A^{**},\Box  )$.
\item $A^{**} = I\oplus A$, as a topological direct sum of closed subspaces $I$ and $A$, with
$I$ as an ideal in $A^{**}$,  the projection $p$ on $A^{**}$ onto $A$ being a homomorphism.
\end{enumeratei}
\end{theorem}

\begin{proof}
(i) This follows from Theorem \ref{le:150212A}(ii) and Theorem \ref{th:070911B}(iii).

(ii) Let $A_*$ be a predual of $A$, and consider the canonical Banach space decomposition
$A^{**}=(A_*)^\circ \oplus A  $ (cf.\ Dales \cite[p.~241]{Dales00}).
 It is easy to verify that  $A_*$  is a topologically introverted
subspace of $A^*$, and clearly the two Arens products coincide on $(A_*)^*=A$.
Hence (ii) follows from Theorem \ref{th:070911B}(iii).

(iii) Since $A$ has a bounded two-sided approximate identity  we have the decomposition
$A^{**}\cong (A^*\cdot A)^\circ \oplus (A^*\cdot A)^*$; and since $A$ is a right ideal
in its second dual, we have $WAP(A^*)=A^*\cdot A$ (cf.\ \cite[Corollary 1.2, Theorem 1.5]{BLP98}).
Therefore, $A^{**}\cong WAP(A^*)^\circ  \oplus WAP(A^*)^*$.
It is easy to check that under these conditions, $WAP(A^*)$ is faithful, and therefore
our result follows by Corollary \ref{co:150212D}  and Theorem \ref{th:070911B}(iii).

(iv) An appeal to Theorem \ref{th:070911B} gives the result. 
\end{proof}

Next we study trivolutions on the Banach algebra $L^\infty(G)^*$, equipped with its first Arens product $\Box.$ Here $G$  is a locally compact group. For simplicity of notation, in the following we shall denote $E\Box F$ by $EF$, whenever $E,F\in L^\infty(G)^*$. If $K\subset G$ is  measurable and  $f\in L^\infty  (G)$, let 
\[
\|f\|_K=\esssup { \{|f(x)|\colon x\in K\}},
\]
and let $L_0^\infty (G)$ be the closed ideal of $L^\infty (G)$ consisting of all $f\in L^\infty (G)$ such that
for given $\epsilon >0$,  there exists a compact set $K \subset G$ such that  $\|f\|_{G\setminus K}<\epsilon$.

 In \cite[Theorems 2.7 and  2.8]{LaPy90}, Lau and Pym showed that $L_0^\infty (G)$ is a faithful, topologically introverted
 subspace of $L^\infty(G)$ and  $L^\infty (G)^*$  is the Banach space direct sum
\begin{equation}\label{eq:010412A}
L^\infty(G)^*=L_0^\infty (G)^\circ \oplus L^\infty_0(G)^*.
\end{equation}
In this decomposition  $L_0^\infty (G)^*$ is identified with the closed subalgebra of $L^\infty (G)^*$ defined as the norm closure
of  elements in $L^\infty (G)^*$  with compact carriers ($F\in L^\infty (G)^*$ has compact carrier if for some compact set $K$,
$F(f)=F(\chi_Kf)$ for every $f\in L^\infty (G)$).  In addition, Lau and Pym showed that if
$\pi \colon L^\infty (G)^*\lra LUC(G)^*$ is the natural restriction map, then   $\pi (L_0^\infty (G)^*)=M(G)$.
Lau and Pym \cite{LaPy90} make a case for the study of $L_0^{\infty}(G)^{\ast}$ for general $G$ (in place of $L^1(G)^{\ast \ast}$).
In \cite{Singh99}, the third named author  has expressed $L_0^{\infty}(G)^{\ast}$ as the second dual of $L^1(G)$
with a locally convex topology similar to the strict topology (see also \cite{GrLo84}).
 Let  $\mathscr E(G)$ denote  the set of all right identities of $L^\infty (G)^*$, and
$\mathscr E_1(G)$ the set of those with norm  one.  In $L^\infty (G)^*,$ when $G$ is not discrete,
there is an abundance of such right identities, a fact noted and well-utilized in
(\cite{Grosser84}, \cite{LaPy90}, \cite{Singh99}, \cite{FaGh07}),  for instance.

For the convenience of our  readers, we shall now state the following result of Lau and Pym (\cite{LaPy90},  Theorems 2.3 and 2.11)
which will be needed  in the last three theorems of this paper. In the following results all products are with respect to the first Arens product.

\begin{theorem}[Lau--Pym]\label{th:140911A}
Let $G$ be a locally compact group and let the map $\pi \colon L^\infty (G)^*\lra LUC(G)^*$ be the natural restriction map. Then
\begin{enumeratei}
\item $\mathscr E_1(G)\subset L_0^\infty (G)^*$.
\item For each $E\in \mathscr E (G)$, $\restrict{\pi}{EL^\infty (G)^*}$ is a continuous isomorphism  from $EL^\infty (G)^*$ to $LUC (G)^*$,
and if $\|E\|=1$, the isomorphism is an isometry.
\item For each $E\in \mathscr E_1(G)$,  $L_0^\infty (G)^*=E L_0^\infty (G)^* + (\ker \pi \cap L_0^\infty (G)^*)$,
and the algebra $EL_0^\infty (G)^*$ is isometrically isomorphic with $M(G)$ via $\pi$.
\end{enumeratei}
\end{theorem}

\begin{theorem}\label{th:150911A}
Let $G$ be a non-discrete locally compact group and  $X$ and $Y$ be subalgebras of $L^\infty (G)^*$ with $L_0^\infty (G)^*\subset Y \subset X$.
Then there are no trivolutions of $X$ onto $Y$. In particular, $L^\infty(G)^{*}$ has no trivolutions with range
$L_0^\infty (G)^*$.
\end{theorem}

\begin{proof}
If $G$ is compact, then $X=Y=L^\infty(G)^*$, and hence any trivolution of $X$ onto $Y$ is 
an involution on $L^\infty(G)^*$.  Such an involution does not exist if $G$ is non-discrete by
Grosser \cite[Theorem 2]{Grosser84}.

Let $G$ be non-compact. To obtain a contradiction, let $\rho$ be a trivolution from $X$ onto $Y$. 
By  Theorem \ref{th:140911A}(i), $\mathscr E_1(G)\subset Y$.
By Theorem \ref{ex:120911B}(i),  each $E$ in $\mathscr E_1(G)$ is the identity for $Y$, and therefore, also the identity
for $L_0^\infty (G)^*$.  But the identity for $L_0^\infty (G)^*$ is clearly in the topological centre of $L_0^\infty (G)^*$
and so it belongs to $L^1(G)$ (cf.\ Budak--I\c{s}\i k--Pym \cite[Proposition 5.4]{BIP11}), which is not possible since $G$ is not discrete.
\end{proof}

\begin{theorem}\label{co:150911B}
The algebra $L_0^\infty (G)^*$ has an involution if and only if $G$ is discrete. Further, if $G$ is discrete, $L^\infty (G)^*$
has a trivolution with range  $L^1(G)$, extending the natural involution on $L^1(G)$.
\end{theorem}

\begin{proof}
If $G$ is discrete, then $L_0^\infty (G)^*=C_0(G)^*=L^1(G)$  has a natural involution, and hence  by \eqref{eq:010412A}
and by Theorem \ref{th:070911B}(iii),
$L^\infty (G)^*$ has a trivolution with range $L^1(G)$, extending the involution of $L^1(G)$.

If $G$ is not discrete, then the result follows from Theorem \ref{th:150911A} upon  taking $X=Y=L_0^\infty (G)^*$.
\end{proof}

\begin{remark}
In the above theorem we used the natural involution of $\ell^1(G)$ to obtain a trivolution on $\ell^\infty(G)^*$
with range  $\ell^1(G)$.  Alternatively, we can  view $\ell^\infty (G)^*$ as the Banach algebra  $M(\beta G)$,
where $\beta G$ is the Stone--$\check{\text{C}}$ech compactification of an infinite discrete group $G$,  recalling  that
$M(\beta G)=M(\beta G\setminus G)\oplus \ell^1(G)$ where $M(\beta G\setminus G)$ is an ideal in $M(\beta G)$
(Dales--Lau--Strauss \cite[Theorem 7.11 and  (7.4)]{DLS10}) and then use Theorem \ref{th:300112B}(iv).
\end{remark}

\begin{theorem}\label{th:150911C}
If $G$ is compact, then for each $E\in \mathcal E(G)$, there are trivolutions of $L^\infty(G)^*$ onto $EL^\infty (G)^*$.
\end{theorem}

\begin{proof}
Let $E\in \mathscr E(G)$. The compactness of $G$ implies that $L_0^\infty (G)=L^\infty (G)$ and
$LUC(G)^*=M(G)$. Let $\rho$ be any involution on $LUC(G)^*$ and let $\pi '=\restrict{\pi}{EL^\infty (G)^*}$.
It follows from Theorem \ref{th:140911A}(iii) that $\rho ':=(\pi ')^{-1}\circ \rho \circ \pi '$ is an involution on $EL^\infty (G)^*$.
Let $\ell_E \colon L^\infty (G)^* \lra EL^\infty (G)^*$ be the left multiplication by $E$.
 Then by Theorem \ref{ex:120911B}(iii), $\tau :=\rho '\circ \ell_E$ is a trivolution of $L^\infty (G)^*$ onto $EL^\infty (G)^*$,
 as required.
\end{proof}

\begin{theorem}\label{th:150911D}
Let $G$ be a locally compact group. For each $E\in \mathscr E_1(G)$, there exists a trivolution
of $L_0^\infty (G)^*$ onto $EL_0^\infty (G)^*$.
\end{theorem}

\begin{proof}
By Theorem \ref{th:140911A}, $\mathscr E_1(G)\subset L_0^\infty(G)^*$, and for each $E\in \mathscr E_1(G)$,
$EL_0^\infty (G)^* \cong M(G)$. If $\rho$ is an involution on $M(G)$, then it is easily checked that
$\rho ':=(\restrict{\pi}{EL_0^\infty (G)^*})^{-1}\circ \rho \circ (\restrict{\pi}{EL_0^\infty (G)^*})$ is an
involution on $EL_0^\infty (G)^*$, and  hence by Theorem \ref{ex:120911B}(iii),
$\tau :=\rho '\circ \ell_E$ is a trivolution of $L_0^\infty (G)^*$ onto $EL_0^\infty (G)^*$.
\end{proof}

\begin{remark}\label{re:160911A}
 If $\rho$ in the proofs of Theorems \ref{th:150911C} or \ref{th:150911D}  restricts to an involution  $\rho_0$ on $L^1(G)$, then  in view
of the fact that $\pi$ is the identity on $L^1(G)$, the trivolution $\tau$  constructed in the respective proofs
will be  an extension of $\rho_0$.
\end{remark}


The  authors would like to thank Kenneth A. Ross for his comments and useful discussions on the topic.
They also thank  F.\ Ghahramani for his comments. 
The second author was partially supported by NSERC.
 The third author thanks Indian National Science Academy for support under the INSA Senior Scientist Programme and Indian Statistical Institute, New Delhi, for a Visiting Professorship under this programme together with excellent research facilities.

\providecommand{\bysame}{\leavevmode\hbox to3em{\hrulefill}\thinspace}
\providecommand{\MR}{\relax\ifhmode\unskip\space\fi MR }
\providecommand{\MRhref}[2]{%
  \href{http://www.ams.org/mathscinet-getitem?mr=#1}{#2}
}
\providecommand{\href}[2]{#2}

\end{document}